\documentclass[a4paper, 12pt,oneside,reqno]{amsart}
\usepackage[a4paper]{geometry}
\usepackage{amssymb,amsmath,amsthm}
\usepackage[T2A]{fontenc}
\usepackage{hyperref}

\usepackage{xcolor}

\def\Nov{\mathrm{Nov}}
\def\Com{\mathrm{Com}}
\def\Lie{\mathrm{Lie}}
\def\Res {\mathop {\fam 0 Res}}

\def\Cur {\mathop {\fam 0 Cur}\nolimits}

\def\Cend {\mathop {\fam 0 Cend}\nolimits}
\def\CDer {\mathop {\fam 0 CDer}\nolimits}

\def\oo#1{\mathbin{{}_{(#1)}}}
\def\ooc#1{\mathbin{{\circ}_{#1}}}

\newtheorem{lemma}{Lemma}
\newtheorem{proposition}{Proposition}
\newtheorem{theorem}{Theorem}
\newtheorem{corollary}{Corollary}
\newtheorem{remark}{Remark}

\theoremstyle{definition}
\newtheorem{example}{Example}
\newtheorem{definition}{Definition}

\title{On the structure of finite Novikov conformal algebras}
\author{Pavel Kolesnikov, Jiefeng Liu}
\address{Sobolev Institute of Mathematics, Novosibirsk (Russia)}
\address{Northeast Normal University, Changchun (China)}

\keywords{Novikov algebra, conformal algebra, simple algebra, semisimple algebra, cohomology}
\subjclass[2020]{17B69, 17D25, 16S99}

\thanks{Supported in part by RSF 25-41-00005, NSFC (W2412041, 12371029) and the National Key Research and Development Program of China (2021YFA1002000).}

\begin{document}

\begin{abstract}
We  classify finite simple Novikov conformal algebras over an algebraically closed field of characteristic zero. The classification consists of the current conformal algebra over the base field  and a one-parameter list of Virasoro-like conformal algebras.
A semisimple finite Novikov conformal algebra is proved to be a direct sum of simple ones.
We also describe deformations and central extensions of finite simple Novikov conformal algebras.
\end{abstract}

\maketitle

\section{Introduction}

The classification of finite simple conformal algebras 
is a foundational problem in the structural theory initiated by 
V.~Kac~
\cite{KacVABeginn}. A Lie conformal algebra encodes the singular part of the operator product expansion in two-dimensional conformal field theory; from the algebraic point of view, it may be regarded as an algebra in the pseudo-tensor category of modules over the polynomial ring $H=\Bbbk[\partial]$~\cite{BDK2001}. Following Kac's work, the complete descriptions of simple finite Lie and associative conformal algebras were obtained in~\cite{DK1998}, the Jordan case was settled in~\cite{Zelm2000} and further developed in~\cite{KR-2008}, and the corresponding classifications of Lie conformal superalgebras followed in~\cite{FK-2002,FKR-2004}. In this well-established landscape, the Novikov variety stands out as a natural and physically motivated class that  has remained largely unexplored from the structural point of view.

Novikov algebras first arose in the study of Hamiltonian differential operators~\cite{GelDorfm} and generalized Poisson brackets in hydrodynamic-type equations~\cite{BalNov}; they are deeply intertwined with the algebraic structures underlying two-dimensional conformal field theory~\cite{BDK2001,BK2002}. A Novikov algebra is a linear space $V$ equipped with a bilinear operation $\circ$ satisfying the left-symmetric and right-commutative identities
\[
(a\circ b)\circ c - a\circ (b\circ c) = (b\circ a)\circ c -b\circ (a\circ c), 
\quad 
(a\circ b)\circ c = (a\circ c)\circ b,
\]
for all $a,b,c\in V$. The first structure results for the ordinary (non-conformal) case were obtained in~\cite{Zelmanov87}: a finite-dimensional simple Novikov algebra over a field of characteristic zero is a field, and every semisimple one is a direct sum of simple algebras. Further progress in~\cite{Osborn92,ZP2024,Xu2001} culminated in a complete understanding of finite-dimensional simple Novikov algebras as Gelfand--Dorfman constructions of differentiably simple commutative algebras. 

Despite this accumulated understanding, the structure of finite simple Novikov conformal algebras has remained entirely open, leaving a conspicuous gap in the conformal algebra literature.
The present paper fills this gap. 
Our main result is a complete classification of finite simple and semisimple Novikov conformal algebras over an algebraically closed field~$\Bbbk$ of characteristic zero, together with the description of their deformations and central extensions. We prove that every simple finite Novikov conformal algebra is isomorphic either to the current conformal algebra 
$\Cur\Bbbk$ or to a one-parameter family of Virasoro-like conformal algebras $V_{\alpha}$ 
($\alpha\in\Bbbk$); moreover, every semisimple finite Novikov conformal algebra is a direct sum of such simple ones. This completes the structural picture of finite simple conformal algebras initiated by V.~Kac, adding the Novikov variety as the last natural class to the landscape alongside Lie, associative, and Jordan.

Our approach rests on the structure theory of associative conformal algebras with a finite faithful representation developed in~\cite{Kol2006Adv}. The central new ingredient is the analysis of the right multiplication algebra $\mathcal{R}$ and its interaction with conformal derivations. We show that the semisimplicity problem for a finite Novikov conformal algebra can be reduced to that of an associative conformal algebra with a finite faithful representation, and since the latter is fully understood by~\cite{Kol2006Adv}, this reduction yields a practical criterion for detecting semisimplicity and decomposing arbitrary finite Novikov conformal algebras. 
Beyond the classification itself, these criteria are of independent interest and are potentially applicable to a wider class of non-associative conformal algebras. In this sense, the paper not only settles the last piece of the classification problem described above, but also reinforces the general philosophy that finite simple conformal algebras are governed by the representation theory of associative conformal algebras, thereby offering a unified perspective for future investigations.

The paper is organized as follows. In Section~2 we recall the basic definitions of Novikov conformal algebras and introduce the notation used throughout. Section~3 develops the necessary machinery of conformal endomorphisms and establishes the key structural properties of the right multiplication algebra $\mathcal{R}$. Section~4 contains the classification of simple and semisimple finite Novikov conformal algebras. Section~5 is devoted to cohomological issues.
We propose a general approach to define conformal cohomologies for an arbitrary class of algebras corresponding to a binary quadratic operad based on the Chevalley--Eilenberg cohomology of Lie conformal algebras \cite{BKV1999}. Being applied to the operad of Novikov algebras, this method allows us to describe deformations and central extensions via conformal cohomology.

\section{Novikov conformal algebras}

Given an arbitrary non-associative algebra 
$(A,\cdot )$ over a field $\Bbbk $, 
$\mathrm{char}\,\Bbbk =0$, 
a subspace $C\subseteq A[[z,z^{-1}]]$
of the space of formal distributions 
with coefficients in $A$ is called a 
{\em conformal algebra of formal distributions} \cite{KacForDist}
if the following three conditions hold:
\begin{itemize}
    \item [(C1)]
for every $a(z),b(z)\in C$ there exists 
an integer $N\ge 0$ such that 
\[
a(w)b(z)(w-z)^N = 0
\]
in $A[[w,w^{-1},z,z^{-1}]]$;

\item[(C2)]
if $a(z)\in C$ then $\partial_z a(z)\in C$;

\item[(C3)]
if $a,b\in C $ then the formal distributions---coefficients 
of the polynomial 
\[
(a\oo\lambda b)(z) =\Res\limits_{w=0} a(w)b(z)\exp (\lambda (w-z))
\]
(at $\lambda^n$, $n\ge 0$) 
belong to $C$.
\end{itemize}

Conformal algebras of formal distributions are models of the abstract theory of conformal algebras.
By definition \cite{KacVABeginn}, 
a {\em conformal algebra} $C$
is a linear space equipped with a linear map $\partial$ 
and with a polynomial-valued bilinear operation $(\cdot \oo\lambda \cdot )$ satisfying the {\em sesqui-linearity} condition:
\begin{equation}\label{eq:sesqui-lin}
(\partial a\oo\lambda b) = -\lambda (a\oo \lambda b),
\quad 
(a\oo\lambda \partial b) = (\partial+\lambda) (a\oo \lambda b).
\end{equation}

For every (abstract) conformal algebra $C$ one may construct an ``ordinary''  algebra $A=\mathcal A(C)$ such that $C$
embeds into $A[[z,z^{-1}]]$ as a conformal algebra of formal distributions. 
If $\mathcal A(C) $ is associative (Lie, commutative, etc.) then $C$ is said to be an associative (resp., Lie, commutative, etc.) conformal algebra. In fact, for a variety $\mathfrak V$ of ordinary algebras 
a conformal algebra $C$ is a $\mathfrak V$-conformal algebra if and only if there exists an algebra $A$ from $\mathfrak V$ such that $C$ is isomorphic to a conformal algebra of formal distributions over~$A$.

Every Lie conformal algebra may be presented by formal distributions over an ``ordinary'' Lie algebra. 

\begin{example}\label{exmp:Vir}
The Virasoro conformal algebra $\mathrm{Vir}$
is the free $\Bbbk[\partial ]$-module generated by a single element $v$ such that 
\[
(v\oo\lambda v)(z) = \partial v(z) +2\lambda v(z).
\]
As a conformal algebra of formal distributions, 
$\mathrm{Vir}$
is generated by the single series 
$v(z) = \sum\limits_{n\in \mathbb Z}
t^n\partial_t z^{-n-1}$
over the Witt Lie algebra $\mathrm{Der}\,\Bbbk [t,t^{-1}]$.
\end{example}

\begin{example}\label{exmp:Current}
Given an ``ordinary'' algebra $(P,\cdot)$,
construct the free module 
$C=\Bbbk [\partial ]\otimes P$ and define 
\[
(x\oo\lambda y) = xy,\quad x,y\in P,
\]
then extend to all $C$ by
\eqref{eq:sesqui-lin}. Then $C$ is a conformal algebra called {\em current} conformal algebra over $P$ denoted $\Cur P$. In this case, 
$\mathcal A(C) = P[t,t^{-1}]$, 
so  $\Cur P$ is a $\mathfrak V$-conformal algebra if and only if $P$ belongs to~$\mathfrak V$.
\end{example}

An alternative but equivalent approach to the definition of what is a $\mathfrak V$-conformal algebra was proposed in \cite{BDK2001}. 
A conformal algebra may be considered 
as an algebra in the pseudo-tensor category 
$H$-mod for $H=\Bbbk [\partial ]$. If we denote by $\mathcal O_{\mathfrak V}$ the operad governing the variety $\mathfrak V$ of algebras \cite{GK1994} then the class of $\mathfrak V$-conformal algebras coincides with the class of all morphisms from $\mathcal O_{\mathfrak V}$ to $H$-mod.

A conformal algebra is said to be {\em finite} if it is finitely generated as a module over $H=\Bbbk [\partial ]$.
An ideal $I$ of a conformal algebra $C$ is an $H$-submodule which is closed with respect to $(\cdot\oo\lambda \cdot)$, i.e., 
$(I\oo\lambda C) + (C\oo\lambda I)\subseteq I[\lambda ]$.
A conformal algebra $C$ is said to be {\em simple} if 
$(C\oo\lambda C)\ne 0$ and there are no 
non-zero proper ideals of~$C$.

If $\mathfrak V$ is a 2-variety (i.e., square of an ideal is again an ideal; so are the varieties of Lie, associative, and Novikov algebras) then, 
given an ideal $I$ in a $\mathfrak V$-algebra, 
it is natural to consider the descending series of ideals $I^{(1)}=I$, $I^{(n+1)} = I^{(n)}I^{(n)}$. 
An ideal $I$ is {\em solvable} if $I^{(n)}=0$ for some $n\ge 1$. An algebra without nonzero solvable ideals is called {\em semisimple}. 
For $\mathfrak V$-conformal algebras,  the definition of semisimplicity is completely similar.

A simple finite Lie conformal algebra $V$ is isomorphic either  to 
$\Cur \mathfrak g$ for a simple finite-dimensional Lie algebra $\mathfrak g$ 
or to the Virasoro conformal algebra \cite{DK1998}. In the associative or Jordan cases
(\cite{Zelm2000}, see also \cite{KR-2008})
only current conformal algebras emerge in the classification of simple finite ones. The purpose of this paper is to describe simple and semisimple finite Novikov conformal algebras.

Let $\Nov$ be the class of all Novikov algebras over the field $\Bbbk $. Denote by $\mathcal O_{\Nov}$ 
the corresponding operad. Then a Novikov conformal algebra is a morphism from $\mathcal O_{\Nov }$ to 
the pseudo-tensor category $\Bbbk [\partial ]$-mod. 
Being translated to the language of $\lambda $-products, this definition turns into the following one.  

\begin{definition}[\cite{Hong2015}]
A Novikov conformal algebra is a left (unital) module $V$ over $H=\Bbbk [\partial ]$
equipped with a sesqui-linear operation 
\[
(\cdot \ooc{(\lambda)} \cdot ): V\otimes V \to 
\Bbbk [\partial,\lambda ]\otimes _H V\cong V[\lambda ]
\]
such that the following identities hold for all 
$a,b,c\in V$:
\begin{equation}\label{eq:LSym}
(a\ooc{(\lambda)} b)\ooc{(\lambda +\mu)} c - a\ooc{(\lambda)} (b\ooc {(\mu)} c)
=
(b\ooc{(\mu)} a)\ooc{(\lambda +\mu)} c - b\ooc{(\mu)} (a\ooc{(\lambda)} c),
\end{equation}
\begin{equation}\label{eq:RCom}
(a\ooc{(\lambda)} b)\ooc\mu c = \{(a\ooc{(\lambda)} c)\ooc{(\mu-\lambda )} b\}.
\end{equation}
Hereinafter, we use the notation
\[
\{x\ooc{(\lambda )} y\} = (x\ooc{(-\partial-\lambda)} y).
\]
\end{definition}

\begin{example}\label{exmp:CurrentNov}
If $A$ is a Novikov algebra with a binary product 
$\circ : A\otimes A\to A$ then the current conformal algebra $\Cur A$ (see Example~\ref{exmp:Current}) is a Novikov conformal algebra.
\end{example}

\begin{example}\label{exmp:Valpha_distr}
Let $A=\Bbbk[t,t^{-1}]$ be the (commutative) algebra of Laurent polynomials. Let us fix a scalar $\alpha \in \Bbbk $ and define the new operation $\circ $ on $A$ by the rule 
\[
f(t)\circ g(t) = f(t)g'(t) +\alpha f(t)g(t), 
\]
where $g' = \partial_t g$ is the ordinary derivative. Then $(A,\circ ) $ is a Novikov algebra. Consider the formal distribution 
\[
v(z) = \sum\limits_{n\in \mathbb Z}
 t^nz^{-n-1} \in A[[z,z^{-1}]].
\]
Then 
$v(w)\circ v(z) (w-z)^2 =0$ and 
\[
(v\ooc{(\lambda)} v)(z) = \alpha v(z) -\partial_z v(z) -\lambda v(z),
\]
so the distribution $v(z)$ together with 
all its formal derivatives span a conformal algebra of formal distributions over the Novikov algebra $(A,\circ )$.
\end{example}

\begin{example}\label{exmp:Quadratic}
Let $P$ be a Novikov--Poisson algebra \cite{Xu1993}, 
i.e., 
a linear space with two operations $\circ, *:P\otimes P\to P$
such that $(P,\circ)$ is a Novikov algebra, $(P,*)$ is an associative and commutative algebra, and the following 
identities hold:
\[
(x\circ y)*z = (x*z)\circ y, 
\quad 
x*(y\circ z) - x*(z\circ y)  = y\circ (x*z) - z\circ (x*y).
\]
Then the free $H$-module $H\otimes P$ 
equipped with the $\lambda $-product 
\[
(1\otimes x)\ooc{(\lambda)} (1\otimes y ) = 1\otimes (x\circ y) 
+\lambda (1\otimes (x*y)) + \partial \otimes (x*y), 
\quad x,y\in P,
\]
is a Novikov conformal algebra.
\end{example}

In particular, for a fixed scalar $\alpha \in \Bbbk $,
the 1-dimensional space 
$P=\Bbbk v$ is a Novikov--Poisson algebra relative to the operations 
\[
v\ast v= v,\quad v\circ v=\alpha v.
\]
The Novikov conformal algebra 
constructed according to Example~\ref{exmp:Quadratic} 
is denoted $\mathcal V_\alpha $, 
it is isomorphic to the algebra from Example~\ref{exmp:Valpha_distr}.

\begin{example}\label{exmp:Differential}
Let $C$ be an associative and commutative conformal algebra 
with operations $\partial $ and $(\cdot\oo\lambda \cdot)$. 
Suppose $D:C\to C$ is a derivation of $C$, i.e., 
a $\partial $-invariant linear map 
such that 
$D(a\oo\lambda b ) = D(a)\oo\lambda b + a\oo\lambda D(b)$
for all $a,b\in C$. 
Fix a scalar $\alpha\in \Bbbk $ and define 
 the new $\lambda$-product on 
 the same $H$-module  $C$:
\[
a\ooc{(\lambda)} b = a\oo\lambda D(b) +\alpha (a\oo\lambda b) , \quad a,b\in C,
\]
is a Novikov conformal algebra denoted $C^{(D,\alpha)}$. 
\end{example}

For example, $D=\partial $ is a derivation in the above-mentioned sense, so every commutative conformal algebra $C$ 
is a Novikov conformal algebra relative to the new operation 
\[
a\ooc{(\lambda)} b = (\partial +\lambda +\alpha)(a\oo\lambda b), \quad a,b\in C.
\]

\begin{remark}
The notion of a derivation on a conformal algebra mentioned above should not be confused with {\em conformal derivation}, which will be discussed later.
\end{remark}

In contrast to ordinary Novikov algebras over a field, there exist Novikov conformal algebras that cannot be embedded 
into a commutative conformal algebra with a derivation. 
However, every finitely generated Novikov conformal algebra 
may be presented as a subalgebra of $C^{(D,0)}$ for an appropriate 
commutative conformal algebra $C$ with a derivation~$D$ \cite{KolNest_arxiv}.

\section{Conformal endomorphisms and the conformal Burnside Theorem}

Let $H=\Bbbk[\partial ]$ be the polynomial algebra as above, and let $V$ be an $H$-module.
A {\em conformal endomorphism} $f$  of $V$ \cite{KacVABeginn} is a linear map 
\[
f_\lambda : V\to \Bbbk[\partial,\lambda ]\otimes_H V\cong 
V[\lambda ], 
\quad 
\]
such that 
\[
f_\lambda (\partial v) = (\partial+\lambda) f_\lambda (v), \quad v\in V.
\]
If $V$ is a finitely generated $H$-module then 
the space $\Cend V$ of all conformal endomorphisms of $V$
is an associative conformal algebra relative 
to the operations 
\[
(\partial f)_\lambda (v) = -\lambda f_\lambda (v),
\]
\[
(f\oo\lambda g)_\mu (v) = f_\lambda (g_{\mu-\lambda }(v)), 
\]
for $f,g\in \Cend V$, $v\in V$.

If $V$ is the free $H$-module of rank $n\ge 1$
then $\Cend V$ is often denoted $\Cend_n$. 
The structure of this associative conformal algebra
was systematically studied in \cite{BKL2003}. 
Let us state here an isomorphic presentation 
following \cite{Kol2006Adv}. 

An element of $V$ may be presented by a column 
of polynomials, i.e., $V\cong H\otimes \Bbbk ^n$.
A conformal endomorphism is presented by a matrix 
$a(\partial, x) \in M_n(\Bbbk [\partial , x])$, so that 
\[
a(\partial, x)_\lambda (h(\partial)\otimes u) = 
h(\partial+\lambda )a(-\lambda , \partial) u,
\]
for $h\in H$, $u\in \Bbbk ^n$.
Hence, the associative conformal algebra structure 
on $\Cend_n = M_n(\Bbbk [\partial, x])$
is given by 
\[
a(\partial, x)\oo\lambda b(\partial, x)
 = a(-\lambda , x)b(\partial+\lambda , x+\lambda ).
\]

Suppose $\mathcal S $ is a conformal subalgebra 
of the associative conformal algebra $\Cend_n$, 
and let $U$ be an $H$-submodule 
of $V=H\otimes \Bbbk^n$. 
Then $\mathcal S(U)\subseteq V$ stands for the linear span of all 
coefficients (at $\lambda ^s$, $s\ge 0$)
of all $f_\lambda (u)$, $f\in \mathcal S$, $u\in U$.
Since the base field $\Bbbk $ is infinite, we may replace the collection of coefficients with the collection of values, i.e.,
\begin{equation}\label{eq:Cend(U)_defn}
    \mathcal S(U) = \mathrm{span}_{\Bbbk} \{ f_\alpha(u) \mid f\in \mathcal S, u\in U, \alpha\in \Bbbk \}.
\end{equation}
This is obviously an $H$-submodule of~$V$.

A submodule $U$ of $V$ is said to be 
$\mathcal S$-{\it invariant}
if $\mathcal S(U)\subseteq U$.
If there are no nontrivial $\mathcal S$-invariant submodules in $V$ then the conformal subalgebra
$\mathcal S \subseteq \Cend_n$ is said to be 
{\em irreducible} \cite{BKL2003}.

As an example, consider 
$\Cur_n = M_n(\Bbbk [\partial ])\subset \Cend_n$, 
this is an irreducible conformal subalgebra.
Another series of examples may be constructed as 
follows: choose a matrix $Q(x)\in M_n(\Bbbk [x])$, 
$\det Q\ne 0$, 
and consider 
$\Cend_{n,Q}
= M_n(\Bbbk [\partial, x])Q(x-\partial )$.
This is also an irreducible conformal subalgebra 
of $\Cend_n$.

In \cite{BKL2003} it was conjectured that the images of $\Cur_n$ under automorphisms of $\Cend_n$ and 
conformal subalgebras of the form $\Cend_{n,Q}$,
$\det Q\ne 0$, exhaust all irreducible conformal subalgebras of $\Cend_n$. This conjecture was proved in \cite{Kol2006Adv}, and this result led to the structure theory of associative conformal algebras
with finite faithful representation. 

Namely, we say that an associative conformal algebra
$\mathcal S$
has a finite faithful representation (FFR) if 
$\mathcal S$ is isomorphic to a subalgebra of 
$\Cend V$ for some finitely generated $H$-module~$V$.

\begin{theorem}[\cite{Kol2006Adv}]
\label{thm:FFRClassif}
Let $C$ be an associative conformal algebra
with an FFR. 
\begin{enumerate}
\item If $C$ is simple then $C$ is isomorphic
 either to $\Cur_n$ or to 
 $\Cend_{n,Q}$, for some $n\ge 1$ 
 and $Q\in M_n(\Bbbk[x])$, $\det Q\ne 0$.
 
\item If $C$ is semisimple then $C$ is a finite direct sum of simple ones.

\item If $C$ is not semisimple then there exists a
maximal nilpotent ideal (or nilpotent radical) $N$ of $C$ such that 
$C/N$ is a semisimple conformal algebra with an FFR.
\end{enumerate}
\end{theorem}

Suppose $V$ is a finite nonassociative conformal algebra with a $\lambda $-product 
$(\cdot \ooc{(\lambda)} \cdot ): 
V\otimes V\to V[\lambda ]$.
Consider the following conformal endomorphisms 
of $V$ as of an $H$-module:
\[
L^a_\lambda : v\mapsto a\ooc{(\lambda )} v,
\quad 
R^a_\lambda : v\mapsto \{v\ooc{(\lambda)} a\},
\]
for all $a,v\in V$. 
Denote by $\mathcal L$ the conformal subalgebra 
of $\Cend V$ generated by all $L^a$, 
and let $\mathcal R$ stand for the subalgebra 
of $\Cend V$
generated by all $R^a$, $a\in V$. 
Similarly, denote by $\mathcal M$ 
the subalgebra generated jointly by $\mathcal L$
and $\mathcal R$ in $\Cend V$.

Note that a conformal $\mathcal M$-submodule of $V$
is exactly an ideal of $V$. Hence, if $V$ is simple
then $\mathcal M$ is an irreducible conformal subalgebra of $\Cend V$, hence, $\mathcal M$ 
is a simple associative conformal algebra 
with an FFR.

A conformal endomorphism $d\in \Cend V$ is called a 
{\em conformal derivation} 
of the conformal algebra $V$
if 
\begin{equation}\label{eq:CDer-definition}
    d_\lambda (a\ooc {(\mu)} b) 
    = (d_\lambda (a) \ooc{(\lambda +\mu )} b)
    + (a\ooc{(\mu )} d_\lambda (b)),
\end{equation}
for all $a,b\in V$.

For example, if $V$
is an associative conformal algebra then 
for every $a\in V$ the 
operation of commutator 
$d_\lambda(x)  = [a\ooc{(\lambda)} x]
= (a\ooc{(\lambda )}x) - \{x\ooc {(\lambda)} a\}$, 
$x\in V$,
is a conformal derivation.

In general, the space $\CDer V$ of all conformal derivations 
of a (not necessarily associative) conformal algebra 
$V$ is a Lie conformal subalgebra of $\Cend V$.

\begin{example}\label{exmp:CDerCur}
The space  $\CDer\Cur_1$
is an $H$-module of rank one.
\end{example}

Indeed, suppose $d\in \CDer\Cur_1$, $\Cur_1=He$, 
$e\oo\lambda e =e$, 
$d_\lambda (e) = f(\partial, \lambda )e$.
Then 
$f(\partial,\lambda )e
= d_\lambda (e\oo\mu e)
= (f(\partial,\lambda)e\oo{\mu+\lambda } e) + 
(e\oo\mu f(\partial,\lambda )e)
= f(-\mu-\lambda, \lambda )e
+ f(\partial+\mu, \lambda )e.
$
The equation 
\[
f(\partial,\lambda ) - f(\partial+\mu, \lambda ) = 
f(-\mu-\lambda, \lambda)
\]
holds only for 
$f(\partial,\lambda ) =
(\partial+\lambda )h(\lambda )$, 
so it depends only on one 
polynomial $h\in H$. 
Therefore,  the polynomial presentation of $d\in \Cend_1$ is
\[
d = h(-\partial )(x-\partial), \quad h\in H,
\]
so the single generator of $\CDer\Cur_1$ over $H$ 
is $v = x-\partial $.
As a Lie conformal algebra, 
$\CDer\Cur_1$ is isomorphic to the Virasoro 
conformal algebra:
\[
[v\oo\lambda v] = (v\oo\lambda v) - \{v\oo\lambda v\} = (2\lambda +\partial )v.
\]

Recall that an algebra (or a conformal algebra) is called weakly Noetherian if it satisfies the ascending chain condition for two-sided ideals.

\begin{lemma}\label{lem:DerivNilp}
Let $C$ be a weakly Noetherian associative conformal algebra 
and let $I$ be a nilpotent ideal of $C$. 
Then for every $\varphi \in \CDer(C)$ the $H$-submodule 
$I+\varphi(I)$ is a nilpotent ideal of $C$.
\end{lemma}

\begin{proof}
Let $a\in I$, $\alpha \in \Bbbk $, then 
$\varphi_\alpha (a) = \varphi_\lambda (a)|_{\lambda=\alpha }\in C$.
The $H$-submodule $\varphi(I)$ is the linear span of all 
such elements.
For every $x\in C$ we have 
\[
x\oo\lambda \varphi_\alpha (a)
=\varphi_\alpha (x\oo\lambda a) - \varphi_\alpha(x)\oo{\lambda +\alpha } a \in (\varphi_\alpha(I)+ I)[\lambda ].
\]
In a similar way, $\varphi_\alpha(a)\oo\lambda x \in (I+\varphi_\alpha (I))[\lambda ]$, so $I+\varphi_\alpha (I)$ is an ideal of $C$.

Now suppose $I$ is nilpotent, i.e., there 
exists $n\ge 1$ such that 
\[
I^n_{\lambda_1,\dots, \lambda_{n-1}}
= 
I\oo{\lambda_1}(I\oo{\lambda_2} ( \dots (I\oo{\lambda_{n-1}} I)\dots )) = 0
\]
in $C[\lambda_1,\dots , \lambda _{n-1}]$.
Then for every fixed $\alpha \in \Bbbk $
the expression 
for $\varphi_\alpha ^n(I^n_{\lambda_1,\dots, \lambda_{n-1}}) = 0$ leads us to 
\[
(\varphi_\alpha(I) )^n_{\lambda_1+\alpha,\dots ,\lambda_2+\alpha } \subseteq I.
\]
Hence, 
$I+\varphi_\alpha (I)$ is a nilpotent ideal 
of index $\le n^2$.

If $I_0=I$ is not $\varphi$-invariant then 
there exists $\alpha _1\in \Bbbk $
such that $I\subsetneq I_1 = I+\varphi_{\alpha_1}(I)$. 
If the nilpotent ideal $I_1$ is not $\varphi$-invariant then there exists 
$\alpha _2\in \Bbbk $ such that 
$I_1\subsetneq I_2=I_1+\varphi_{\alpha_2}$, 
and so on. 
We obtain an ascending chain of nilpotent ideals $I_k$, 
$k\ge 0$.
Since $C$ is Noetherian, 
there exists a number $N\ge 0$ such that 
$I_N$ is $\varphi$-invariant. 
In particular, $I+\varphi(I)\subseteq I_N$,
thus $I + \varphi(I)$ is nilpotent.
\end{proof}

\begin{proposition}\label{prop:Conf_Comm_Novikov}
For a finite Novikov conformal algebra $V$, 
the generators $L^a$, $R^a $
($a,b\in V$)
of the corresponding 
conformal subalgebra $\mathcal M\subseteq \Cend V$
satisfy the following relations:
\begin{gather}
(R^a\oo\lambda R^b) = \{R^b\oo\lambda R^a\}, 
                             \label{eq:R-R_commute}\\
(R^a\oo\lambda L^b) =  L^{\{b\ooc\lambda a\}},
                       \label{eq:R-L_product} \\
[L^a\oo\lambda R^b] = R^{(a\ooc\lambda b )} 
  - \{R^b\oo\lambda R^a\}. 
                                       \label{eq:L-R_commute}
\end{gather}
\end{proposition}

\begin{proof}
On the one hand, the identities \eqref{eq:LSym} and \eqref{eq:RCom} imply, respectively, 
\begin{gather}
  a\ooc{(\lambda)} \{ c\ooc{(\mu)} b\} - \{(a\ooc{(\lambda)} c)\ooc{(\mu )}b\}
   = \{c\ooc{(\lambda +\mu)} \{ a\ooc{(\mu)} b \}\}
   - \{ \{ c\ooc{(\lambda)} a \} \ooc{(\mu)} b \},    \label{eq:LSym-R}\\
 \{\{ c\ooc{(\mu)} b\}\ooc{(\lambda)} a \} =   
  \{\{ c\ooc{(\lambda)} a\}\ooc{(\mu)} b \}        \label{eq:RCom-R}
\end{gather}
in the same way as in \cite{KacForDist}. 
These identities are equivalent to \eqref{eq:L-R_commute}
and \eqref{eq:R-R_commute}. The identity \eqref{eq:R-L_product}
is a different form of \eqref{eq:RCom}.

On the other hand, a finite Novikov algebra can always be embedded into an associative and commutative conformal algebra 
with an ordinary derivation $D$ in such a way that 
$a\ooc{(\lambda)} b = a\oo\lambda Db$. In this way, for example, 
the equation \eqref{eq:LSym-R} is easy to prove:
\begin{multline*}
a\ooc{(\lambda )}\{ c\ooc{(\mu)} b\} - \{(a\ooc{(\lambda)} c)\ooc{(\mu)} b\}
=
a\oo{\lambda} D\{c\oo\mu Db \} - \{(a\oo\lambda Dc)\oo\mu Db\} \\
=
a\oo\lambda \{ c \oo\mu D^2 b\} 
= 
a\oo\lambda (D^2b \oo\mu c),
\end{multline*}
\begin{multline*}
\{c\ooc{(\lambda +\mu)} \{ a\ooc{(\mu )}b \}\}
   - \{ \{ c\ooc{(\lambda)} a \} \ooc{(\mu)} b \}
=
\{c\oo{\lambda +\mu} D\{ a\oo\mu Db \}\}
   - \{ \{ c\oo\lambda Da \} \oo\mu Db \} \\
=
(D(Db\oo\mu a)\oo{\lambda +\mu} c) - (Db\oo\mu(Da\oo\lambda c))
=
((D^2b\oo\mu a)\oo{\lambda +\mu } c) 
= D^2b\oo\mu (a\oo\lambda c).
\end{multline*}
The right-hand sides are equal since 
$x\oo\lambda (y\oo\mu z) = y\oo\mu (x\oo\lambda z)$
in a commutative conformal algebra.
The equation \eqref{eq:RCom-R} can be checked similarly.
\end{proof}

In particular, the elements $L^a\in \Cend V$ determine derivations 
$[L^a\oo\lambda \cdot ]$ of $\Cend V$,
and it follows from \eqref{eq:L-R_commute} that the subalgebra 
$\mathcal R$ 
is invariant under these derivations. 
Suppose $\mathcal J$ is an ideal of $\mathcal R$.
If $[L^v\oo{\lambda} \mathcal J] \subseteq \mathcal J[\lambda ]$ for all $v\in V$ then we say $\mathcal J$ is {\em ad-invariant}.

\begin{lemma}\label{lem:R-left-ideal_Nov}
Let $V$ be a finite Novikov conformal algebra. 
Suppose $I$ is a left ideal of $V$ and 
$\mathcal J$ is an ad-invariant 
ideal of the commutative conformal algebra 
$\mathcal R$.
Then $\mathcal J(I)$ is a two-sided ideal of~$V$.
\end{lemma}

\begin{proof}
We just need to show 
that 
$L^v_\lambda (\mathcal J(I)) \subseteq \mathcal J(I)[\lambda ]$ since 
$\mathcal J(I)$ is already closed with respect to 
right multiplications.

Suppose $f \in \mathcal J$, $a\in I$, $v\in V$.
Then 
\[
L^v_\lambda (f_\mu(a))
= (L^v \oo\lambda f)_{\mu+\lambda }(a) 
= \{f\oo\lambda L^v \}_{\mu+\lambda }(a) 
+ [L^v\oo\lambda f]_{\lambda +\mu }(a).
\]
The coefficients of the second summand belong 
to $\mathcal J(I)$.
The first summand 
$\{f\oo\lambda L^v \}_{\mu+\lambda }(a)$
belongs to $\mathcal J(I)[\lambda,\mu ]$ since
\[
\{f\oo\lambda L^v \}_{\mu+\lambda }(a)
=
(f\oo{-\partial -\lambda} L^v )_{\mu+\lambda }(a)
=
(f\oo{\mu} L^v )_{\mu+\lambda }(a)
=
f_\mu (L^v_\lambda (a))
\]
and $L^v_\lambda (a)\in I[\lambda ]$.
\end{proof}

\begin{corollary}\label{cor:R-ideal-Nov}
If $\mathcal J = \mathcal R$ and $I\ne 0$ then either 
$I$ is an ideal of $V$ such that $I\ooc{(\lambda )}V=0$, or
$\mathcal R(I)$ is a nonzero ideal of $V$.
\end{corollary}

\begin{lemma}\label{lem:R-ideal-square}
Let $V$ be a finite Novikov conformal algebra and let
$\mathcal J$ be an ad-invariant ideal of $\mathcal R$.
Then $\mathcal J(V)\ooc{(\lambda )} \mathcal J(V) 
\subseteq \mathcal J^2(V)[\lambda ]$.
\end{lemma}

Note that if $\mathcal J\ne 0$ then $\mathcal J(V)$ is a nonzero ideal of $V$.

\begin{proof}
For every $f,g\in \mathcal J$, $u,v\in V$, $\alpha,\beta,\gamma\in \Bbbk $ we have to show 
\[
f_\alpha(u)\ooc{(\gamma )} g_\beta (v) =
L^{f_\alpha(u)}_\gamma g_\beta (v)\in \mathcal J^2(V).
\]
It follows from
\eqref{eq:R-L_product} that 
\begin{equation}\label{eq:R-L_Cend}
L^{f_\alpha(u)}_\gamma = f_\alpha L^u_{\gamma-\alpha } \in \Cend V.
\end{equation}
Indeed, a generic element of $\mathcal R$ may be presented as
\[
f = \sum \partial^s(R^{a_1}\oo{\alpha_1}(R^{a_2}\oo{\alpha_2}\dots (R^{a_n}\oo{\alpha_n} R^{a_{n+1}})\dots )).
\]
Then 
\[
f_\alpha(u) = \sum (-\alpha)^s R^{a_1}_{\alpha_1}R^{a_2}_{\alpha_2}\dots R^{a_n}_{\alpha_n}R^{a_{n+1}}_{\alpha-\alpha_1-\dots -\alpha_n}(u),
\]
and consecutive application of \eqref{eq:R-L_product} 
leads us to \eqref{eq:R-L_Cend}.

Hence, 
\begin{multline*}
f_\alpha(u)\ooc{(\gamma )} g_\beta (v)
=f_\alpha L^u_{\gamma-\alpha}g_\beta (v)
= f_\alpha g_\beta L^u_{\gamma-\alpha }(v)
+f_\alpha [L^u_{\gamma-\alpha }, g_\beta](v)  \\
=
(f\oo\alpha g)_{\beta+\alpha}(u\ooc{(\gamma-\alpha)} v)
+
(f\oo{\alpha} [L^u\oo{\gamma-\alpha}g])_{\beta+\gamma}(v).
\end{multline*}
Both summands in the right-hand side of this expression belong to $\mathcal J^2(V)$.
\end{proof}

\begin{remark}
Note that if $\mathcal J$ is an ad-invariant ideal of $\mathcal R$ then so is $\mathcal J^2$. Therefore, if $\mathcal J$ is nilpotent then $\mathcal J(V)$ is solvable.
\end{remark}

\begin{proposition}\label{prop:R-Semisimple}
If $V$ is a finite (semi)simple Novikov conformal algebra
then $\mathcal R$ is a semisimple subalgebra of $\Cend V$.
\end{proposition}

\begin{proof}
Assume the converse: there exists a nonzero nilpotent ideal in $\mathcal R$.
Then the nilpotent radical $\mathcal N$ 
(i.e., the maximal nilpotent ideal)
of $\mathcal R$ is nonzero. 

Since $\mathcal R$ is generated by the operators of right multiplication in $V$ and the latter is a finite conformal algebra, we may conclude that $\mathcal R$ is a finitely generated commutative conformal algebra. Hence, $\mathcal R$ is Noetherian \cite{Kol2001_envelope}. 
By Lemma~\ref{lem:DerivNilp}, the maximal nilpotent ideal 
$\mathcal N$ of $\mathcal R$ is invariant with respect to all conformal derivations of $\mathcal R$.

In particular, it works for 
$\varphi_\lambda = [L^u\oo\lambda \cdot ]$,
$u\in V$.
By Lemma~\ref{lem:R-left-ideal_Nov}
$\mathcal N(V)$ is an ideal of $V$
which is solvable by Lemma~\ref{lem:R-ideal-square}.

Therefore, $\mathcal N=0$ as desired.
\end{proof}

\section{Simple and semisimple finite Novikov conformal algebras}

Throughout the section, let $V$ be a simple 
finite Novikov conformal algebra. 
Then $V$ is a torsion-free $H$-module since 
the $H$-torsion is always an annihilator ideal in a conformal algebra (see, e.g., \cite{KacVABeginn}).

Suppose $\rho : V\to \mathcal R\subseteq \Cend V$
is the right regular representation of $V$, i.e., 
\[
\rho(a) = R^a, \quad a\in V.
\]
This is an $H$-linear map. Denote by $A$ the kernel of $\rho $, i.e., the space of all $a\in V$ such that $V\ooc{(\lambda)} a =0$.

\begin{lemma}\label{lem:RAnn=ideal}
If $a\in A$, $a\ne 0$ then $\mathcal R(Ha)=V$.
\end{lemma}

\begin{proof}
If $a\in A$ then $Ha $ is a left ideal of $V$:
$V\ooc{(\lambda)} a = 0$.
By Lemma~\ref{lem:R-left-ideal_Nov}, $\mathcal R(Ha)$
is an ideal of $V$, so 
either $\mathcal R(Ha)=0$ or $\mathcal R(Ha)=V$.
If the first option works for at least one 
$A\ni a\ne 0$
then the space of all such $a\in A$
is an annihilator ideal of $V$.
Therefore,  $\mathcal R(Ha)=V$
for all $A\ni a\ne 0$.
\end{proof}

\begin{corollary}\label{cor:RankOne}
The conformal algebra $\mathcal R$ is isomorphic 
to $\Cur_1 = \Cur\Bbbk $.
\end{corollary}

\begin{proof}
As a semisimple associative conformal algebra with 
FFR, $\mathcal R$ has to be isomorphic to a direct sum of simple ones according to Theorem~\ref{thm:FFRClassif}. Since 
$\mathcal R$ is commutative by Proposition~\ref{prop:Conf_Comm_Novikov},  
so are all the simple summands, but the only commutative simple conformal algebra with FFR is $\Cur_1$.

Hence, $\mathcal R$ is a direct sum of several copies of $\Cur_1$: 
$\mathcal R = \mathcal R_1\oplus \dots \oplus \mathcal R_m$, $\mathcal R_i\cong \Cur_1$.
Each $\mathcal R_i$ is an ideal of $\mathcal R$
which is closed under all conformal derivations 
of $\mathcal R$, in particular, under the derivations $[L^v\oo\lambda \cdot ]$, $v\in V$.
Hence, 
the $H$-submodule $\mathcal R_i(V)$ is $\mathcal M$-invariant, i.e., $V=\mathcal R_i(V)$ for every 
$i=1,\dots, m$.

Assume $m>1$. Then $V=\mathcal R_1(V)=\mathcal R_2(V)$ implies $V=\mathcal R_1(\mathcal R_2(V))=(\mathcal R_1\mathcal R_2)(V)=0$, 
a contradiction. Hence, $\mathcal R\cong \Cur_1$.
\end{proof}

\begin{theorem}\label{thm:RankOne}
Let $V$ be a simple Novikov conformal algebra 
over an algebraically closed field $\Bbbk $ 
of characteristic zero. Then  $V$ is an $H$-module of rank~1.
\end{theorem}

\begin{proof}
As above, $\mathcal M$ stands for the conformal 
subalgebra of $\Cend V$ generated by all operators of left 
($L^v $) and right ($R^v$) multiplication. 
Recall that if $A=\{a\in V \mid R^a = 0\}$
then for every $0\ne a\in A$
we have $\mathcal R(Ha)=V$ by Lemma~\ref{lem:RAnn=ideal}.

Consider the map
\[
\ell : V\to \CDer\mathcal R, \quad v\mapsto [L^v\oo\lambda \cdot ].
\]
By Example \ref{exmp:CDerCur}, the rank of the $H$-module 
$\CDer \mathcal R$
is equal to 1.

Therefore, we have two $H$-linear maps 
$\rho : V\to \mathcal R\cong \Cur_1$
and 
$\ell: V\to \CDer \mathcal R\cong \mathrm{Vir}$.
If the rank of $V$ is greater than 2 then the intersection 
of two kernels $A =\ker \rho $ and $\ker\ell$ is nonzero since each of them is of corank~1.

Assume $0\ne a\in A\cap \ker\ell$. Then 
$\mathcal R(Ha)=V$, so 
\[
L^a_\lambda (V) = L^a_\lambda (\mathcal R(Ha))
=\mathcal R(L_\lambda^a Ha) = 0 
\]
since $a\ooc{(\lambda )} a =0$. 
In this case, $a$ belongs to the two-sided annihilator of $V$, a contradiction to simplicity of~$V$.

Hence, the rank of $V$ does not exceed~2.
Assume the rank of $V$ is equal to~2.
Then there exists a basis $e_1,e_2$ of $V$ over $H$
such that $\rho (e_2)=0$, i.e., $R^{e_2} =0$.
The conformal algebra structure on $V$ is then completely defined by $R^{e_1}\in \Cend_2$. 
Suppose 
\[
R^{e_1} = \begin{pmatrix}
    a(x,x-\partial) & f(x,x-\partial ) \\
    b(x,x-\partial ) & g(x,x-\partial )
\end{pmatrix},
\]
where $a,b,f,g\in \Bbbk [x,\partial ]$. 
In particular, 
$\{e_2\ooc{(\lambda)} e_1\} = R^{e_1}_\lambda (e_2) = f(\partial,\partial+\lambda )e_1 + g(\partial, \lambda +\partial)e_2$,
so 
$e_2\ooc{(\lambda )} e_1 = f(\partial,-\lambda )e_1+g(\partial, -\lambda )e_2$.
Hence,  
\[
L^{e_1} = 
\begin{pmatrix}
    a(x,\partial ) & 0 \\
    b(x,\partial) & 0
\end{pmatrix},
\quad 
L^{e_2} = 
\begin{pmatrix}
    f(x,\partial) & 0 \\
    g(x,\partial) & 0
\end{pmatrix}.
\]
Note that $f\ne 0$: otherwise, 
$He_2$ is a non-trivial ideal of $V$.

By \eqref{eq:L-R_commute}, 
$[L^{e_2}\oo\lambda R^{e_1}] = R^{(e_2\ooc{(\lambda)} e_1)} = f(\partial, -\lambda )R^{e_1}$.
The latter leads us to the $2\times 2$ matrix equation
\begin{equation}\label{eq:LR-comm-V2}
L^{e_2}\oo\lambda R^{e_1} - R^{e_1}\oo{-\partial-\lambda } L^{e_2}
= f(\partial, -\lambda ) R^{e_1},
\end{equation}
i.e.,
\begin{multline*}
\begin{pmatrix}
    f(x,-\lambda) & 0 \\
    g(x,-\lambda) & 0
\end{pmatrix}
\begin{pmatrix}
    a(x+\lambda,x-\partial) & f(x+\lambda,x-\partial ) \\
    b(x+\lambda,x-\partial ) & g(x+\lambda,x-\partial )
\end{pmatrix} 
\\
-
\begin{pmatrix}
    a(x,x-\partial-\lambda) & 
    f(x,x-\partial -\lambda) \\
    b(x,x-\partial-\lambda ) &
    g(x,x-\partial-\lambda )
\end{pmatrix}
\begin{pmatrix}
    f(x-\partial-\lambda
    ,-\lambda) & 0 \\
    g(x-\partial-\lambda ,-\lambda) & 0
\end{pmatrix}
\\
=
f(\partial,-\lambda )
\begin{pmatrix}
    a(x,x-\partial) & f(x,x-\partial ) \\
    b(x,x-\partial ) & g(x,x-\partial )
\end{pmatrix}.
\end{multline*}
Consider the (1,2)-component of this matrix equation:
\[
f(x,-\lambda ) f(x+\lambda,x-\partial ) 
= f(\partial,-\lambda ) f(x,x-\partial).
\]
This is a quadratic equation for a polynomial 
$f(x,y)$ in two formal variables $x$, $y$. 
Let us make a substitution 
$\partial \to x-\partial $ to get an equivalent 
equation
\begin{equation}\label{eq:12-f-equation}
f(x,-\lambda ) f(x+\lambda,\partial ) 
= f(x-\partial,-\lambda ) f(x,\partial).
\end{equation}
It is obvious that $f(x,y)=q(y)$
is a solution of this equation for every 
$q\in \Bbbk [y]$. Let us show that there are no 
other solutions.

Assume $f(x,y)\ne 0$ is a solution of \eqref{eq:12-f-equation}. Then put $\lambda =0$ 
to obtain $f(x,0)=f(x-\partial, 0)$.
The latter means $f(x,0)=c\in \Bbbk $. 
Expand left- and right-hand sides of 
\eqref{eq:12-f-equation}
by the Taylor's formula and compare the terms at 
$\lambda $:
\[
\big (
c - \lambda f'_y(x,0) +\lambda^2(\dots )
\big )
\big (
f(x,\partial ) + \lambda f'_x(x,\partial ) +
\lambda ^2(\dots )
\big ) \\
=
f(x,\partial )
\big (
c -\lambda f'_y(x-\partial, 0) 
+\lambda^2(\dots ) 
\big ),
\]
\[
c f'_x(x,\partial ) = f(x,\partial )f'_y(x,0) -f(x,\partial)f'_y(x-\partial, 0).
\]
If $c\ne 0$ then we obtain that $f(x,y)$ divides 
its derivative
$f'_x(x,y)$ 
which is only possible for $f'_x(x,y)=0$, i.e., 
$f(x,y)=q(y)$ as desired. 
If $c=0$ then $f(x,y)=y \hat f(x,y)$ for some 
$\hat f\in \Bbbk [x,y]$ of smaller degree.
It is obvious that $\hat f$ satisfies the same equation \eqref{eq:12-f-equation}, 
so $\hat f(x,y)$ depends only on $y$ by induction reasoning.
Therefore, $f(x,y)=q(y)\in \Bbbk[y] $,
 so 
\[
R^{e_2} = \begin{pmatrix}
    * & q(x-\partial ) \\ * & *
\end{pmatrix},
\quad q\ne 0.
\]

Recall that $R^{e_1}$ is a generator of the 
conformal subalgebra $\mathcal R\subset \Cend_2$ isomorphic to $\Cur_1$.
Hence, 
\[
R^{e_1} = h(\partial )e,
\]
where $h(\partial )\in H$,
$e=e(x,\partial )\in \Cend_2$ 
is an idempotent: $e\oo\lambda e = e$. Moreover, we know from 
Lemma \ref{lem:RAnn=ideal} that $\mathcal R(He_2)=V$. 
Since $(e\oo{0} \cdot )=\mathrm{id}_{\mathcal R}$, 
we conclude 
$e(x,0)=I$, the identity 
$2\times 2$-matrix,
$e_{12}(x,0)=0$. 
Therefore, 
\[
R^{e_2}_0 = h(0) e(\partial, 0) \in M_2(H),
\]
but 
$h(\partial) e_{12}(x,\partial ) = f(x,x-\partial) = q(x-\partial )$
implies $h = 0$, 
a contradiction. 

As a result, the hypothesis $\mathrm{rank}\,(V) \ge 2$ 
contradicts to the simplicity of $V$, so the only option is 
$\mathrm{rank}\,(V)=1$.
\end{proof}

\begin{proposition}\label{eq:Rank-1-classification}
Let $V = Hv$ be the free $H$-module of rank one 
equipped with a sesqui-linear $\lambda $-bracket
\[
v\ooc{(\lambda)} v = f(\partial,\lambda ) v.
\]
Then $V$ is a Novikov algebra if and only if 
$f=\beta(\partial+\lambda ) + \alpha$,
$\alpha,\beta \in \Bbbk $.
\end{proposition}

\begin{proof}
For the ``if'' part, it is enough to check 
\eqref{eq:LSym} and \eqref{eq:RCom}
for $a=b=c=v$ which is straightforward (see also 
\cite{Hong2015}).

For the ``only if'' part, suppose 
\[
v\ooc{(\lambda)} v = f(\partial,\lambda )v
\]
for some nonzero polynomial $f(x,y)\in \Bbbk [x,y]$.
Then the right commutativity \eqref{eq:RCom}
is equivalent to 
\begin{equation}\label{eq:RCom-rank1}
f(-\mu, \lambda ) f(\partial , \mu)
=
f(\partial+\mu-\lambda, \lambda )f(\partial, -\partial-\mu+\lambda ).
\end{equation}
Let us prove that \eqref{eq:RCom-rank1} holds if and only if 
$f(x,y)$ is a polynomial in $x+y$. The ``if'' part is obvious.
To show ``only if'', put $f(x,y)=q(x,x+y)$ for some polynomial $q(x,y)$.
Then 
\[
q(-\mu, \lambda -\mu)q(\partial, \partial+\mu)=q(\partial+\mu-\lambda ,\partial+\mu)q(\partial, \lambda-\mu).
\]
Make a substitution 
$\partial=z$, $\partial+\mu =x$, 
$\lambda -\mu=y$ to obtain an equation for $q$:
\begin{equation}\label{eq:RCom-Q}
    q(z-x,y)q(z,x)=q(z-y,x)q(z,y).
\end{equation}
Put $y=0$ into \eqref{eq:RCom-Q} to get
$q(z-x,0)q(z,x) = q(z,x)q(z,0)$. If $q\ne 0$
then $q(x,0)=c_0\in \Bbbk $.

Next, evaluate the derivative $\partial/\partial y$ of 
\eqref{eq:RCom-Q} at $y=0$:
\[
q'_2(z-x,0)q(z,x)=-q'_1(z,x)q(z,0) + q(z,x)q'_2(z,0),
\]
where $q'_1$ and $q'_2$ stand for the partial derivatives of 
$q(x_1,x_2)$ relative to $x_1$ and $x_2$, respectively.

As a result, 
\[
(q'_2(z,0)-q'_2(z-x,0))q(z,x)= c_0 q'_1(z,x) ,
\]
i.e., either $c_0=0$ or 
$q$ divides its derivative $q'_1$. The second option means 
that $q(x,y)$ depends only on $y$, as desired.
If $c_0=0$ then 
\[
q(x,y) = y \hat q(x,y),
\]
where $\hat q$ satisfies the same equation \eqref{eq:RCom-Q}.
Hence, $q(x,y)$ depends only on $y$, and $f(x,y) = q(x+y)$.

The left symmetry 
\eqref{eq:LSym} turns into the following equation on $f(\partial,\lambda )$:
\[
f(-\lambda -\mu,\lambda )f(\partial, \lambda+\mu)
-
f(\partial+\lambda, \mu)f(\partial,\lambda )
=
f(-\mu -\lambda,\mu )f(\partial, \lambda+\mu)
-
f(\partial+\mu, \lambda )f(\partial,\mu ).
\]
For $f(x,y)=q(x+y)$, we have 
\[
q(-\mu )q(\partial+\lambda+\mu)
-
q(\partial+\lambda+\mu)q(\partial+\lambda )
=
q( -\lambda )q(\partial+\lambda+\mu)
-
q(\partial+\mu+\lambda )q(\partial+\mu ).
\]
Since $q\ne 0$, we have 
\[
 q(\partial+\mu ) + q(-\mu ) = q(\partial+\lambda )
+ q( -\lambda ) .
\]
The latter means that $q(y) = \alpha + \beta y$, 
$\alpha,\beta \in \Bbbk $, as desired.
\end{proof}

Denote by $\mathcal V_\alpha $
the rank one Novikov conformal algebra 
$Hv$ with 
\[
(v\ooc{(\lambda )} v) = (\partial+\lambda +\alpha) v,
\]
for $\alpha \in \Bbbk $.

\begin{corollary}\label{cor:Classification}
A simple finite Novikov conformal algebra $V$ 
over an algebraically closed field of characteristic zero is isomorphic either
to $\Cur_1$ or to $\mathcal V_\alpha$.
\end{corollary}

As it was shown in \cite{DK1998}, it may happen that a semisimple finite Lie conformal algebra is not isomorphic to the direct sum of simple ones. Therefore, in contrast to the case of ``ordinary'' algebras, a semisimple Lie conformal algebra may be not {\em classically semisimple}. 
The purpose of this section is to show that for finite Novikov conformal algebras every semisimple object is a direct sum of ideals each isomorphic to a simple Novikov conformal algebra from Corollary \ref{cor:Classification}.

Throughout this section, $V$ is a finite semisimple Novikov conformal algebra, $\mathcal R$ is the conformal subalgebra of $\Cend V$ generated by all right multiplications $R^a$, $a\in V$. By Proposition~\ref{prop:R-Semisimple}, $\mathcal R$ is semisimple, so we may conclude 
$\mathcal R\cong \Cur_1 \oplus \dots \oplus \Cur_1$, i.e.,
\[
\mathcal R = H e^1 \oplus \dots \oplus He^k,
\quad e^i\in \mathcal R,\ (e^i\oo\lambda e^j) = \delta_{ij}e^j.
\]

\begin{lemma}\label{lem:V^2=V}
The conformal algebra $V$ is perfect, i.e., $V^2=V$.
\end{lemma}

\begin{proof}
By definition, $V^2=\mathcal R(V)$. Assume there exists 
$a\in V$ but $a\notin V^2$. 
Then 
\[
\bar a = a -e^1_0(a)-\dots - e^k_0(a) \notin V^2
\]
since $e^i_\alpha (a)\in \mathcal R(V) = V^2$.
In particular, $\bar a\ne 0$.

Suppose 
$f\in \mathcal R$,
\[
f = h_1(\partial )e^1+\dots +h_k(\partial )e^k, \quad h_i\in H.
\]
Then 
$f_\alpha (a) =\sum\limits_{i=1}^k h_i(-\alpha )e^i_\alpha (a)$,
$f_\alpha (e^j_0 (a)) = (f\oo\alpha e^j)_\alpha (a) = h_j(-\alpha )e^j_\alpha (a)$.
Hence, 
$f_\alpha(\bar a)=0$ and, in particular,
$R^v_\lambda (\bar a)=0$, so $\bar a\ooc{(\lambda )}V=0$.
The left annihilator of $V$ is an Abelian ideal of $V$, which cannot be nonzero.
\end{proof}

\begin{lemma}\label{lem:Ideal-Ann}
    Let $I$ be a proper ideal of $V$. 
Then $\mathrm{Ann}_l (I) = \{ v\in V \mid
v\ooc{(\lambda )}I = 0\}$ is nonzero ideal of~$V$.
\end{lemma}

\begin{proof}
It is clear that 
$\mathrm{Ann}_l(I)$ is an ideal of $V$. 
Indeed, 
\[
\mathrm{Ann}_l(I) = \{u\in V \mid R^a_\lambda (u) = 0, a\in I\}.
\]
It follows from \eqref{eq:R-R_commute} and \eqref{eq:L-R_commute}
that 
if $u\in \mathrm{Ann}_l(I)$ then so are 
$R^v_\alpha (u)$ and $L^v_\alpha (u)$ for all $v\in V$, $\alpha\in \Bbbk $.

Denote by $\mathcal R^I$ the ideal of $\mathcal R$ generated by all $R^a$, $a\in I$. Then $\mathcal R^I(V)\subseteq I\ne V$, in particular, $\mathcal R^I\ne \mathcal R$ by Lemma~\ref{lem:V^2=V}. 

Every proper ideal in 
$\mathcal R\cong \Cur_1\oplus \dots \oplus \Cur_1$ 
has a nonzero annihilator (at least, one of $e^i$s). 
Suppose $\mathcal R^I\oo\lambda e = 0$ for some $0\ne e\in \mathcal R$. Then for every $\alpha,\beta \in \Bbbk $
and for every $a\in I$, $v\in V$ we have 
\[
0=(R^a\oo\alpha e)_{\beta+\alpha }(v)
 = R^a_\alpha (e_\beta (v)).
\]
Since $e\ne 0$, there exist $\beta $ and $v$ such that 
$u = e_\beta(v)\ne 0$, this is an element of the left annihilator of $I$. 
\end{proof}

\begin{theorem}
A finite semisimple Novikov conformal algebra $V$ is isomorphic to a direct sum of simple ones:
\[
V = V_1\oplus \dots \oplus V_k,
\]
$V_i\cong \Cur_1$ or $V_i\cong \mathcal V_{\alpha_i}$,
$\alpha_i\in \Bbbk$, $i=1,\dots, k$.
\end{theorem}

\begin{proof}
Proceed by the rank of $V$ as $H$-module. If the rank is zero then $V=0$ (since torsion elements annihilate $V$). For $V\ne 0$, find a maximal ideal $I$ of $V$.
By Lemma~\ref{lem:Ideal-Ann} there exists 
$J=\mathrm{Ann}_l(I)\ne 0$.
Since $I\cap J$ is an Abelian ideal of $V$ and $I$ is maximal, we conclude $V = I\oplus J$, $J\cong V/I$.
If the rank of $J$ is zero then it belongs to the torsion of $V$ as $H$-module, which is impossible.
Also, an ideal of $I$ is an ideal of $V$.

Hence, the rank of $J$ is positive and $I$ is a semisimple Novikov conformal algebra, the rank of $I$ is smaller than that of $V$. 
Therefore, $I = V_2\oplus \dots \oplus V_k$, each $V_i$ is simple, and the claim follows for $V_1=J\cong V/I$.
\end{proof}

\section{Deformations and extensions}

An infinitesimal conformal deformation of
a Novikov conformal algebra $V$ with a $\lambda$-product
$(\cdot \ooc{(\lambda )}\cdot )$ may be defined
in a similar way to that of Lie, associative \cite{BKV1999}, or Poisson conformal algebras \cite{LiuZhou2023}.
Denote by $\hbar$ a formal variable.
Then
$\Bbbk \subset \bar \Bbbk_\hbar :=\Bbbk +\Bbbk \hbar \cong \Bbbk[[\hbar ]]/\hbar^2 \Bbbk [[\hbar ]]$.
The $H$-module $V$ may be considered as a subspace
of $\bar V_\hbar := V[[\hbar ]]/\hbar^2 V[[\hbar ]]$.

An {\em infinitesimal conformal deformation}  of $V$ is an
$\bar\Bbbk_\hbar$-linear operation
\[
 u\mathbin{\circ_{(\lambda )}^\hbar } v = u\ooc{(\lambda )} v + \hbar f_\lambda (u,v),
\]
which turns $\bar V_\hbar $ into a Novikov conformal algebra over
$\bar \Bbbk_\hbar $. The latter condition can be written as a collection
of conditions on $f_\lambda : V\otimes V \to V[\lambda ]$:
\begin{equation}
 f_\lambda (\partial u, v) = -\lambda f_\lambda(u,v),\quad
f_\lambda (u, \partial v) = (\lambda +\partial) f_\lambda(u,v);
 \label{eq:sesqui-cochain}
 \end{equation}
 \begin{multline}
 u\ooc{(\lambda )} f_\mu(v,w) - f_{\lambda+\mu}(u\ooc{(\lambda)}v, w) + f_\lambda (u, v\ooc{(\mu)} w) - f_\lambda(u,v)\ooc{(\lambda+\mu)} w \\
{}
=
v\ooc{(\mu)} f_\lambda(u,w) - f_{\lambda+\mu}(v\ooc{(\mu )}u, w) + f_\mu (v, u\ooc{(\lambda)} w) - f_\mu(v,u)\ooc{(\lambda+\mu)} w ;\label{eq:LSym-cochain}
\end{multline}
\begin{equation}
  f_{\lambda+\mu} (u\ooc{(\lambda)} v, w) + f_\lambda(u,v)\ooc{(\lambda+\mu)}w
=
\{ f_\lambda(u,w)\ooc{(\mu)} v \} + f_{-\partial-\mu}(u\ooc{(\lambda)} w, v).
\label{eq:RCom-cochain}
\end{equation}
Denote by $Z^2(V,V)$ the space of all such $f$ satisfying
\eqref{eq:sesqui-cochain}--\eqref{eq:RCom-cochain}.

Two infinitesimal conformal deformations $\bar V_\hbar$ and $\bar V_\hbar'$
 of a Novikov conformal algebra $V$ are {\em equivalent} if there exists a
$\bar \Bbbk _\hbar $-linear map $\varphi :  \bar V_\hbar = V+\hbar V \to \bar V_\hbar' = V+\hbar V$ preserving the deformed operations
such that
\[
 \varphi (u+\hbar v) = u +\hbar (v+\tau(u)), \quad u,v\in V,
\]
where $\tau : V\to V$ is an $H$-linear map.
In other words, $\varphi = \mathrm {id}+\hbar \tau \pmod {\hbar ^2}$.

For example, if the deformations $\bar V_\hbar $ and
$\bar V_\hbar'$ are respectively defined by cocycles $f,f'\in Z^2(V,V)$
then these deformations are equivalent if and only if
\[
 f_\lambda(u,v) - f'_\lambda(u,v) = u\ooc{(\lambda )}\tau (v) - \tau(u\ooc{(\lambda )} v) + \tau(u)\ooc{(\lambda )} v = : (d^1\tau)_\lambda (u,v),
\]
for all $u,v\in V$ (similar to \cite[Theorem 5.14]{LiuZhou2023}).

Denote by $B^{2}(V,V)$ the space of all those $f\in Z^2(V,V)$
that are of the form $d^1\tau$ for some $H$-linear $\tau : V\to V$.
Therefore, the deformations of $V$ up to equivalence
are described by the elements of the space $Z^2(V,V)/B^2(V,V)$.

This ``naive'' approach to conformal cohomology of Novikov algebras
has a natural explanation in terms of operads.

Let us define the cohomology of a Novikov conformal algebra $V$
in the same way as it was done for Lie conformal algebras in \cite[S.\,12]{BKV1999} ({\em BKV-cohomology}, for short).

The structure of a Novikov conformal algebra on an $H$-module $V$
is defined by a morphism $\nu $ from the operad $\mathcal O_{\Nov }$
to the operad $\mathrm{Chom}_V$ of conformal endomorphisms of $V$ (see \cite[S.\,5]{BK-et-al2018}), 
a full subcategory of the pseudo-tensor category $H$-mod:
\[
 \nu : \mathcal O_{\Nov} \to \mathrm{Chom}_V,
 \quad
 \nu(x_1\circ x_2) \mapsto (\cdot \ooc{(\lambda )} \cdot ).
\]

Recall that the Koszul dual operad $\mathcal O_{\Nov }^!$ represents
the class of opposite Novikov algebras (right-symmetric and left-commutative).
Hence,
$\mathcal O_{\Nov}^!$ can be considered as a sub-operad of the operad
$\Omega\Com$ of differential associative and commutative algebras generated by
$x_1\bullet x_2 = x_1' x_2$, where $x'=d(x)$, $d\in \Omega\Com(1)$ is the derivation.
In particular, $\mathcal O_{\Nov}^!(3)$ is spanned by linearly independent
monomials
\[
 x_1''x_2x_3,\ x_1x_2''x_3,\ x_1x_2x_3'',\  x_1'x_2'x_3,\
 x_1'x_2x_3',\ x_1x_2'x_3'.
\]

Consider the canonical morphism  of operads \cite{GK1994}
\[
 \iota :
 \mathcal O_{\Lie} \to \mathcal O_{\Nov }^!\otimes \mathcal O_{\Nov }
\]
sending the generator $[x_1,x_2]\in \mathcal O_{\Lie}(2)$
to
\[
 x_1'x_2\otimes x_1\circ x_2 - x_1x_2'\otimes x_2\circ x_1.
\]
The composition of $\iota $ with $\mathrm{id}\otimes \nu$ defines a morphism
\[
 \mathcal O_{\Lie} \to \mathcal P_V:=\mathcal O_{\Nov }^!\otimes \mathrm{Chom}_V.
\]
Define a complex $C^*(V,V)$ as the reduced BKV-cohomology complex
for the operad $\mathcal P_V$. Namely,
\[
 C^n(V,V) = \{f\in \mathcal P_V(n) \mid f^{\sigma }= (-1)^\sigma f, \ \sigma\in S_n\},
 \quad n\ge 1,
\]
and the differential $d^n: C^n(V,V)\to C^{n+1}(V,V)$ is given by \cite[Definition 12.4]{BKV1999}.
In particular,
for $n=1$ the space $C^1(V,V)$ coincides with $\mathrm{Chom}_V(1)=\mathrm{End}_H(V)$. For $n=2$, an element
$x_1'x_2\otimes f - x_1x_2'\otimes g \in \mathcal P_V(2)$ belongs to
$C^2(V,V)$ if and only if
$g = f^{(12)} \in \mathrm{Chom}_V(2)$. Therefore,
$C^2(V,V) \cong \Cend(V)$.
For $n=3$, the skew-symmetry of
\[
 x_1''x_2x_3\otimes f_1 +x_1x_2''x_3\otimes f_2 +  x_1x_2x_3''\otimes f_3 +
 x_1x_2'x_3' \otimes g_1 + x_1'x_2x_3' \otimes g_2 + x_1'x_2'x_3 \otimes g_3
 \in \mathcal P_V(3),
\]
$f_i,g_i\in \mathrm{Chom}_V(3)$,
means that $f_2 = - f_3^{(23)} = -f_1^{(12)}$ and
$g_2 = -g_3^{(23)} = -g_1^{(12)}$,
so a 3-cochain is defined by a pair $(f_1,g_3)\in \mathrm{Chom}_V(3)^2$.

The differential $d^1$ maps $f\in \mathrm{End}_H(V)$ into
$d^1f\in \Cend(V)$, where
\[
 (d^1f)_\lambda (u,v) = u\ooc{(\lambda )} f(v) - f(u\ooc{(\lambda )} v) + f(u)\ooc{(\lambda )} v
\]
for $u,v\in V$.
Similarly, we may evaluate $d^2(f)$ for $f\in \Cend(V)$.
Consider
\[
 p = x'_1x_2\otimes f - x_1x_2'\otimes f^{(12)} \in \mathcal P_V(2)
\]
and apply the differential formula of the BKV-cohomology theory to
the morphism
\[
 (\mathrm{id}\otimes \nu )\iota: [x_1,x_2]\mapsto \theta : = x_1'x_2\otimes (\cdot\ooc{(\lambda )}\cdot) - x_1x_2'\otimes \{\cdot\ooc{(\lambda )} \cdot\}^{(12)}
\]
to obtain
\begin{equation}\label{eq:Diff-2}
 d^2(p) = \theta (1,p) -\theta (1,p)^{(12)} + \theta (1,p)^{(123)}
 - p(\theta ,1) + p(\theta ,1) ^{(23)} - p(\theta  ,1) ^{(132)},
\end{equation}
where $1$ stands for the identity in $\mathcal P_V(1)$.

\begin{remark}
In a similar way, considering 
the canonical morphism 
\[
\mathcal O^!\otimes \mathcal O,
\]
one may construct cohomology of $\mathfrak V$-conformal algebras for any class $\mathfrak V$ governed by a binary quadratic operad~$\mathcal O$.
\end{remark}

Calculate the composition
$ \theta  (1,p) \in \mathcal P_V(3)$ following the rules of
\cite{BK-et-al2018} in the second tensor factor.
In the non-reduced form we have
\[
 p_{\lambda_1,\lambda_2}(u,v) = x_1'x_2 \otimes f_{\lambda _1}(u,v) - x_1x_2'\otimes f_{\lambda _2}(v,u),
\]
\[
 \theta _{\lambda_1,\lambda_2}(u,v) = x_1'x_2\otimes (u\ooc{(\lambda _1)} v) -x_1x_2'\otimes (v\ooc{(\lambda_2)} u),
\]
where $\lambda_1+\lambda_2 = -\partial $.
Then
\begin{multline}\label{eq:Cohom-comp-1}
\theta (1,p)_{\lambda_1,\lambda_2,\lambda_3}(u,v,w)
=\theta _{\lambda_1,\lambda_2+\lambda_3}(u,p_{\lambda_2,\lambda_3}(v,w))
\\
= x_1'x_2'x_3\otimes (u \ooc{(\lambda_1)} f_{\lambda_2}(v,w))
-x_1'x_2x_3' \otimes (u \ooc{(\lambda_1)} f_{\lambda_3}(w,v)) \\
- x_1(x_2'x_3)'\otimes (f_{\lambda_2}(v,w)\ooc{(\lambda _2+\lambda_3)} u )
+ x_1(x_2x_3')'\otimes (f_{\lambda_3} (w,v) \ooc{(\lambda_2+\lambda_3 )} u),
\end{multline}
where $\lambda_1+\lambda_2+\lambda_3=-\partial $.
Similarly,
\begin{multline}\label{eq:Cohom-comp-2}
 p(\theta ,1)_{\lambda_1,\lambda_2,\lambda_3 }(u,v,w)
 =p_{\lambda_1+\lambda_2,\lambda_3} (\theta _{\lambda_1,\lambda_2}(u,v),w) \\
 =
 (x_1'x_2)'x_3\otimes f_{\lambda_1+\lambda_2}(u\ooc{(\lambda_1)} v, w)
 - (x_1x_2')'x_3 \otimes f_{\lambda_1+\lambda_2}(v\ooc{\lambda_2} u, w)
 \\
 -x_1'x_2x_3'\otimes f_{\lambda_3}(w, u\ooc{(\lambda_1)} v)
 +x_1x_2'x_3'\otimes f_{\lambda_3}(w, v\ooc{(\lambda_2)} u).
\end{multline}
All terms of \eqref{eq:Diff-2} may be evaluated from
\eqref{eq:Cohom-comp-1} and \eqref{eq:Cohom-comp-2} by permutations, for example,
\[
 \theta (1,p)^{(123)}_{\lambda_1,\lambda_2,\lambda_3}(u,v,w)
 =\theta (1,p)_{\lambda_3,\lambda_1,\lambda_2}(w,u,v),
\]
etc. The first tensor factor $x_1''x_2x_3$ appears in
$p(\theta ,1)$, $p(\theta ,1)^{(23)}$, $\theta (1,p)^{(12)}$,  and $\theta (1,p)^{(123)}$.
Collecting the respective summands of \eqref{eq:Diff-2}, we obtain that
$d^2(p)=0$ implies \eqref{eq:RCom-cochain} for $\lambda_1=\lambda$, $\lambda_2=\mu$, $\lambda_3 =-\partial - \lambda -\mu$.
Similarly, comparing the coefficients at $x_1'x_2'x_3$
leads us to \eqref{eq:LSym-cochain}.
Conversely, for every $f\in Z^2(V,V)$, other summands of \eqref{eq:Diff-2}
 turn into zero due to skew symmetry.

Therefore, the second cohomology of $C^*(V,V)$ coincides with
$H^2(V,V) = Z^2(V,V)/B^2(V,V)$ defined above. Let us calculate $H^2(V,V)$
for simple finite Novikov algebras $V=\Cur_1$ and $V=\mathcal V_\alpha$,
$\alpha\in \Bbbk $. As a result, we describe infinitesimal conformal deformations of these Novikov conformal algebras.

In the following statements, $\bar \varphi  = \varphi + B^2(V,V)\in H^2(V,V)$ 
for $\varphi \in Z^2(V,V)$.

\begin{theorem}\label{thm:Z2Cur}
 For $V=\Cur_1=Hv$ with $(v\ooc{(\lambda)} v)=v$ we have $H^2(V,V)=\Bbbk \bar\omega $, where $\omega_\lambda(v,v) = (\partial+\lambda)v$.
\end{theorem}

\begin{proof}
Suppose a cocycle $\varphi \in Z^2(V,V)$ acts on 
the generator of $V$ as 
\[
\varphi_\lambda (v,v) = f(\lambda,\partial ) v
\]
for some polynomial $f(x_1,x_2)\in \Bbbk [x_1,x_2]$.
Due to \eqref{eq:sesqui-cochain} the polynomial 
$f$ completely determines the values of 
$\varphi_\lambda  $ on the $H$-module generated by~$v$.

Relations \eqref{eq:LSym-cochain} and \eqref{eq:RCom-cochain} for $\varphi $ turn into the following polynomial equations on $f$:
\begin{equation}\label{eq:Z2(1)Cur}
    f(\mu, \partial+\lambda ) + f(\lambda , \partial )
    -f(\lambda , -\lambda -\mu)
    =
f(\lambda , \partial+\mu) + f(\mu, \partial )- f(\mu, -\lambda -\mu),
\end{equation}
\begin{equation}\label{eq:Z2(2)Cur}
    f(\lambda+\mu, \partial ) + f(\lambda, -\lambda -\mu) = f(\lambda , \partial+\mu) + f(-\partial-\mu, \partial).
\end{equation}

A cocycle $\varphi $
defined by a polynomial $f$ as above is a coboundary 
if and only if 
\[
\varphi_\lambda (v,v) = v\ooc{(\lambda )} \tau(v) - \tau(v\ooc{(\lambda )} v) + \tau(v) \ooc{(\lambda )} v
\]
for some $H$-linear map $\tau : V\to V$. As $\tau (v) = g(\partial )v $ for some polynomial $g\in \Bbbk [x]$, 
we have 
\begin{equation}\label{eq:B2_Cur}
f(\lambda , \partial ) = (d^1g)(\lambda , \partial ):= g(\partial+\lambda) -g(-\lambda ) - g(\partial ).
\end{equation}

Since the equations \eqref{eq:Z2(1)Cur} and \eqref{eq:Z2(2)Cur} are homogeneous, we may assume 
the polynomial $f(x_1,x_2)$ is homogeneous as well, $m=\deg f\ge 0$.

For $m=0$, a constant polynomial $f(\lambda ,\partial )=c\in \Bbbk $ is a coboundary ($g=c$ works), so there are no non-trivial cocycles in degree zero.

Suppose $m>0$.
Put $\mu=0$ in \eqref{eq:Z2(2)Cur} to get 
\[
f(\lambda ,-\lambda ) = f(-\partial, \partial ), 
\]
i.e., $\lambda +\partial $ divides $f(\lambda , \partial )$.

Put $\mu=0$ in \eqref{eq:Z2(1)Cur} to get 
\begin{equation}\label{eq:Z2(lin)Cur}
f(0,\partial+\lambda ) = f(0,\partial ) - f(0,-\lambda ).    
\end{equation}
If $m=1$ then $f(\lambda ,\partial ) = c_1(\lambda +\partial )$, $c_1\in \Bbbk $. 
This polynomial defines a cocycle which is not a coboundary for $c_1\ne 0$ since no terms of degree one may appear in $d^1g$. Therefore, there is a 1-dimensional space of cohomologies spanned by 
the cocycle $\omega_\lambda (v,v) = (\partial +\lambda ) v$.

Suppose $m>1$. Then 
\[
f(\lambda , \partial ) = a_0\lambda ^m + \dots + a_{m-1}\lambda \partial^{m-1} + a_m\partial ^m.
\]
It follows from \eqref{eq:Z2(lin)Cur} that 
$a_m=0$. 
We may also eliminate the summand with 
$a_{m-1}$ 
modulo an appropriate coboundary: choose 
$g(x) = \frac{1}{m} a_{m-1}x^m$, 
then 
\[
d^1g(\lambda,\partial ) = \lambda (\lambda (\dots )+a_{m-1}\partial^{m-1}).
\]
Hence, modulo $B^2(V,V)$, we may assume that 
$\lambda^2$ divides $f(\lambda , \partial )$.

Apply the partial derivative $\partial /\partial \lambda $ to \eqref{eq:Z2(2)Cur}:
\[
f_1'(\lambda+\mu,\partial ) + f_1'(\lambda , -\lambda -\mu) - f'_2(\lambda , -\lambda -\mu)
=
f'_1(\lambda , \partial+\mu), 
\]
where $f'_i(x_1,x_2)$ stand for $\partial f/\partial x_i$, $i=1,2$.
Put $\lambda =0$ into the last expression. 
The summands $f'_1(0,x)$ are zero since $\lambda^2\mid f$, so we obtain 
\begin{equation}\label{eq:DiffEq_Cur}
f'_1(\mu,\partial )- f'_2(0,-\mu) = 0.
\end{equation}
Finding all homogeneous solutions of \eqref{eq:DiffEq_Cur} is a routine problem on 
the method of indeterminate coefficients, 
the only solution is 
\[
f(x_1,x_2) = c_0x_1^m.
\]
Recall that $x_1+x_2$ divides $f$, so we have $c_0=0$ as required.

Therefore, the only non-trivial cocycle appears in degree one.
\end{proof}

\begin{theorem}\label{thm:Z2Valpha}
 For $V=\mathcal V_\alpha$ with $(v\ooc{(\lambda )} v)=(\partial+\lambda+\alpha )v$ we have
 $H^2(V,V) = \Bbbk \bar \varepsilon$, where $\varepsilon_\lambda (v,v)=v$.
\end{theorem}

\begin{proof}
 Let $\varphi \in Z^2(V,V)$. Then it is completely determined by 
 a polynomial $f(x_1,x_2)\in \Bbbk [x_1,x_2]$
 such that $\varphi_\lambda (v,v) = f(\lambda , \partial )v$
for the generator~$v$ of $\mathcal V_\alpha $.
Relations \eqref{eq:LSym-cochain} and \eqref{eq:RCom-cochain}
turn into 
\begin{multline}\label{eq:Z2V(1)_a}
    (\partial+\lambda+\alpha)f(\mu, \partial+\lambda )- (\alpha-\mu)f(\lambda+\mu,\partial ) + (\partial+\lambda+\mu+\alpha)
     (f(\lambda,\partial ) - 
     f(\lambda , -\lambda -\mu)) \\
    =
(\partial+\mu+\alpha)f(\lambda , \partial+\mu ) - (\alpha -\lambda )f(\lambda+\mu, \partial ) + (\partial+\lambda+\mu+\alpha)
(f(\mu,\partial ) - f(\mu, -\lambda -\mu))
\end{multline}
 and 
\begin{multline}\label{eq:Z2V(2)_a}
(\mu-\alpha )f(\lambda+\mu,\partial )- (\partial+\lambda+\mu+\alpha)f(\lambda, -\lambda-\mu)
\\
=
(\mu-\alpha )f(\lambda , \partial+\mu)
- (\partial+\lambda+\mu+\alpha)f(-\partial-\mu, \partial ).
\end{multline}
A cocycle $\varphi $ defined by a polynomial $f$ is a coboundary 
if and only if 
\[
f(\lambda, \partial ) = (d^1g)(\lambda ,\partial ) :=  (\partial+\lambda+\alpha) (g(\partial+\lambda )+g(-\lambda ) -g (\partial ))
\]
for some polynomial $g\in \Bbbk [x]$.

For example, $f(x_1,x_2)=c \in \Bbbk $ defines a cocycle which is not a coboundary (for $c\ne 0$) since it is not divisible by 
$(\partial+\lambda+\alpha)$.

Let us show that every polynomial $f$ satisfying 
\eqref{eq:LSym-cochain}, \eqref{eq:RCom-cochain} is constant modulo $d^1g$, $g\in \Bbbk [x]$.

Put $\mu=0$ in \eqref{eq:Z2V(2)_a} to obtain 
\[
f(\lambda , -\lambda ) = f(-\partial, \partial ). 
\]
Hence, $f(x,-x) = c$; we may assume $c=0$ modulo the constant cocycle. Therefore, $\partial+\lambda $ divides $f(\lambda, \partial)$.

Then put $\mu=0$ in \eqref{eq:Z2V(1)_a} to get 
\[
(\partial+\lambda+\alpha)(-f(0,\partial+\lambda ) + f(0,\partial)- f(0,-\lambda )) = 0.    
\]
The latter means $f(0,\partial ) = c_1\partial $, $c_1\in \Bbbk $, so 
\begin{equation}\label{eq:Z2V2(h)}
f(\lambda , \partial ) = c_1(\partial+\lambda ) + \lambda(\partial+\lambda )h(\lambda ,\partial)
\end{equation}
for some polynomial $h$.
Let us decompose $h$ into homogeneous summands to present 
\begin{equation}\label{eq:Z2V(3)_a}
f(\lambda , \partial ) = c_1(\lambda +\partial ) + f_2(\lambda,\partial) + \dots + f_m(\lambda, \partial ), 
\end{equation}
where $\deg f_k = k$. Note that $\lambda (\lambda+\partial )$
divides every $f_k(\lambda , \partial )$, $k=2,\dots, m$.

Suppose 
\[
f_k(\lambda, \partial ) = (c_k\lambda \partial^{k-2} + \lambda^2(\dots ))(\lambda +\partial ).
\]
For $k>1$, there exists a polynomial 
$g_k(x)$ such that $d^1g_k(\lambda ,\partial ) = (c_k\lambda \partial^{k-2} + \lambda^2(\dots ))(\lambda +\partial +\alpha )$.
The highest homogeneous component of $d^1g_k$ is of the same form as $f_k$: the coefficient at $\lambda\partial^{k-1}$ is $c_k$. Hence, modulo $B^2(V,V)$, we may assume 
that $\lambda^2$ divides the highest homogeneous component of $f$:
\[
x_1^2 \mid G(x_1,x_2) := f_m(x_1, x_2 ).
\]
The latter means 
$G'_i(0,x_2) = 0$, where $G'_i $ stand for $\partial G/\partial x_i$, $i=1,2$.

Now put the presentation \eqref{eq:Z2V(3)_a}
into \eqref{eq:Z2V(2)_a} and consider the component of highest degree $m+1$:
\begin{equation}
    \label{eq:Z2V(pr)}
    \mu G(\lambda+\mu , \partial ) - (\partial+\lambda+\mu) G(\lambda, -\lambda -\mu) 
    =
    \mu G(\lambda , \partial +\mu ) - (\partial+\lambda +\mu)G(-\partial-\mu,\partial ).
\end{equation}
Apply partial derivative $\partial/\partial\lambda $ to 
\eqref{eq:Z2V(pr)}:
\begin{multline*}
\mu G'_1(\lambda +\mu,\partial ) - G(\lambda , -\lambda -\mu)
-(\partial +\lambda +\mu)(G'_1(\lambda , -\lambda -\mu ) - G'_2(\lambda, -\lambda -\mu )) 
\\
= 
\mu G'_1(\lambda , \partial+\mu) - G(-\partial-\mu, \partial ).
\end{multline*}
Put $\lambda =0$ into the last expression to get
\begin{equation}\label{eq:Z2V(fi)}
\mu G'_1(\mu,\partial ) = - G(-\partial-\mu, \partial ).    
\end{equation}
It is not hard to note that the only homogeneous solution 
$G(x_1,x_2)$ of
\eqref{eq:Z2V(fi)} is proportional to $x_1x_2^{m-1}+x_2^m$
which is not divisible by $x_1^2$ apart from $G=0$.

Therefore, the polynomial $h$ in \eqref{eq:Z2V2(h)} has no nonzero homogeneous components
and $f(\lambda , \partial ) = c_1(\lambda +\partial ) \equiv -\alpha c_1 \pmod {B^2(V,V)}$, 
as required.
\end{proof}

Let $V$ be a Novikov conformal algebra, and 
let $M$ be a Novikov conformal bimodule over 
$V$. 
The second cohomology space of $V$ with coefficients in $M$
describes the equivalence classes of extensions
\[
0\to M\to E\to V\to 0
\]
with $M\ooc{(\lambda )}M = 0$ in $E$.
The details are completely similar to what is known 
in Lie or associative case \cite{BKV1999}.

It is well known that the Virasoro conformal algebra 
$Hv$ with $[v\oo{\lambda } v] = (2\lambda+\partial ) v$
has a non-trivial central extension by means of 
the scalar module $M=\Bbbk e$, $\partial e = 0$.
Let us consider a similar question for simple 
Novikov conformal algebras. 

Let $V$ be a Novikov conformal algebra and let $M=\Bbbk e$ be a 1-dimensional $H$-module such that $\partial e=0$. Then $M$ is a conformal bimodule over $V$ so that 
\[
u\ooc{(\lambda )} e = e\ooc{(\lambda )} u = 0
\]
for all $u\in V$. 
Consider an extension $E$ of $V$ by means of $M$.
The structure of $E$ is completely determined by a mapping 
\[
f_\lambda : V\otimes V \to M[\lambda ]\cong \Bbbk [\lambda ]e
\]
such that 
\begin{gather}
    f_\lambda (\partial u,v) = -\lambda f_\lambda (u,v),
    \quad 
f_\lambda (u,\partial v) = \lambda f_\lambda (u,v),
               \label{eq:sesqui-Ext} \\
f_{\lambda +\mu} (u\ooc{(\lambda )}v, w) - 
f_\lambda (u, v\ooc{(\mu )} w)
=
f_{\lambda+\mu} (v\ooc{(\mu )} u, w) - f_\mu (v, u\ooc{(\lambda )}w), 
         \label{eq:LSym-Ext} \\
f_{\lambda+\mu} (u\ooc{(\lambda )} v, w)
 = f_{-\partial-\mu } (u\ooc{(\lambda )} w, v).
   \label{eq:RCom-Ext}
\end{gather}
The last two equations are obtained from \eqref{eq:LSym-cochain} and \eqref{eq:RCom-cochain}
assuming the action of $V$ on $M$ is trivial.

Two extensions $E$ and $E'$ defined respectively 
by maps $f$ and $f'$ are said to be equivalent 
if there exists an isomorphism of conformal algebras 
$E\to E'$ which is an identity map on~$M$. 
As in the case of Lie or associative conformal algebras
\cite{BKV1999}, $E$ and $E'$ are equivalent if and only if 
\[
f_\lambda (u,v) - f'_\lambda(u,v) = (d^1g)_\lambda (u,v):=
g(u\ooc{(\lambda )} v), \quad u,v\in V,
\]
for some $H$-linear map $g:V\to M$.

\begin{theorem}\label{thm:Nov_Ext}
Let $V=\Cur_1$ or $V=\mathcal V_\alpha $, $\alpha \in \Bbbk $. Then $V$ has no nontrivial extensions by means of the scalar module $M=\Bbbk e$.
\end{theorem}

\begin{proof}
Assume an extension $E$ of $V=Hv$ by means of the scalar module $M$ is defined by a cocycle $\varphi $ such that 
\[
\varphi_\lambda (v,v) = f(\lambda )e
\]
for some polynomial $f(t)\in \Bbbk [t]$. 

Suppose $V=\Cur_1$. Then \eqref{eq:LSym-Ext} turns into
$f(\lambda ) = f(\mu )$, so $f(t)=c_0\in \Bbbk $. 
Obviously, the cocycle $\varphi_\lambda (v,v)=c_0e$
is a coboundary: $\varphi = d^1g$, where $g(v)=c_0e$.

Suppose $V=\mathcal V_\alpha $. Then \eqref{eq:LSym-Ext} and \eqref{eq:RCom-Ext} turn into 
\[
\begin{gathered}
(\lambda -\mu)f(\lambda +\mu ) = (\lambda+\mu+\alpha )(f(\lambda ) - f(\mu )), \\
(\alpha -\mu)f(\lambda +\mu ) = (\lambda+\mu+\alpha )f(-\mu).
\end{gathered}
\]
Since $\lambda -\mu $ and $\lambda+\mu+\alpha $ are mutually prime, we derive that $t+\alpha $ divides $f(t)$. Substitute $f(t)=(t+\alpha )h(t)$ into the second equation to get $h(\lambda+\mu)=h(-\mu)$, i.e., 
$f(t) = (t+\alpha )c_1$, $c_1\in \Bbbk $.
The corresponding cocycle $\varphi $ is a coboundary:
for $g(v)=c_1e$ we have 
$(d^1g)_\lambda (v,v) = g((\partial+\lambda+\alpha)v)=
(\lambda+\alpha)c_1 e$.
\end{proof}

\end{document}